\documentclass[reqno,makeidx,11pt]{amsart}
\usepackage{amssymb,amsfonts,amsmath,graphicx
}

\def\C{{\mathbb{C}}}

\def\E{{\mathbb{E}}}

\def\N{{\mathbb{N}}}
\def\0{{\mathbb{O}}}

\def\Q{{\mathbb{Q}}}
\def\R{{\mathbb{R}}}
\def\T{{\mathbb{T}}}

\def\Z{{\mathbb{Z}}}

\def\cB{{\mathcal B}}
\def\cC{{\mathcal C}}
\def\cD{{\mathcal D}}

\def\cI{{\mathcal I}}

\def\cO{{\mathcal O}}

\newcommand{\norm}[1]{{\left\|{#1}\right\|}}

\newcommand{\abs}[1]{{\left|{#1}\right|}}

\newcommand{\set}[1]{{\left\{{#1}\right\}}}
\newcommand{\un}{{\mathbf 1}}
\newcommand{\dst}{\displaystyle}
\newcommand{\actson}{\curvearrowright}

\def\fig{ \centerline{Fig. \the\count200\global\advance\count200 by 1}}
\newtheorem{thm}{Theorem}[section]
\newtheorem{cor}[thm]{Corollary}
\newtheorem{lem}[thm]{Lemma}
\newtheorem{prop}[thm]{Proposition}

\newtheorem*{thm*}{Theorem}
\newtheorem*{lem*}{Lemma}
\newtheorem*{cor*}{Corollary}

\theoremstyle{definition}
\newtheorem{defn}[thm]{Definition}
\newtheorem{ex}[thm]{Example}

\theoremstyle{remark}
\newtheorem{rem}[thm]{Remark}

\title[Pointwise limits for sequences of orbital integrals]{Pointwise limits for sequences of orbital integrals}

\author{Claire Anantharaman-Delaroche}
\address{Laboratoire de Math\'ematiques et Applications, Physique Math\'ematique d'Orl\'eans (MAPMO - UMR6628),Ê
F\'ed\'eration Denis Poisson (FDP - FR2964),
CNRS/Universit\'e d'Orl\'eans,
B. P. 6759, F-45067 Orl\'eans Cedex 2}
\email{claire.anantharaman@univ-orleans.fr}

\subjclass{Primary  47A35; Secondary 37A15, 28D05, 28C10, 26A42}

\keywords{Ergodic theory, transformation groups, Haar measures, Riemann and Lebesgue integrals, lattices}

\begin{document}

\begin{abstract} In 1967, Ross and Str\"omberg published a theorem about pointwise limits
of orbital integrals for the left action of a locally compact group $G$ onto $(G,\rho)$, where $\rho$ is the right Haar measure.
In this paper, we study the same kind of problem, but more generally for left actions of $G$
onto any measured space $(X,\mu)$, which leaves the $\sigma$-finite measure $\mu$
relatively invariant, in the sense that $s\mu = \Delta(s)\mu$ for every $s\in G$, where $\Delta$
is the modular function of $G$. As a consequence, we also obtain a generalization of a theorem
of Civin, relative to one-parameter groups of measure preserving transformations.

The original motivation for the circle of questions treated here dates back to classical
problems concerning pointwise convergence of Riemann sums relative to Lebesgue integrable functions.

\end{abstract}
\maketitle

\section{Introduction}
\setcounter{equation}{0}
\renewcommand{\theequation}{\thesection.\arabic{equation}}

The study of almost everywhere convergence of Riemann sums is an old problem, with
many ramifications (see \cite{R-W} for a recent  survey). Let us consider the interval $[0,1[$, identified with the torus
$\T = \R/(2\pi \Z)$, and let $\rho$ be the normalized Lebesgue measure on $\T$. Let $f: [0,1[ \to \C$ be a measurable function.
For $n\in \N^*$ and $x\in \T$, the corresponding Riemann sum of $f$ is defined by
$$R_nf(x) = \frac{1}{n} \sum_{j=0}^{n-1} f\big(x+\frac{j}{n}\big).$$
When $f$ is Riemann integrable, for any $x\in \T$ we have
\begin{equation}\label{Riem-sum}
\lim_n R_n(f)(x) = \int_{0}^1 f(t) d\rho(t).
\end{equation}

When $f$ is Lebesgue integrable, it is easily seen that $(R_n(f))$ converges in mean to $\int_{0}^1 f d\rho$ (see for instance
\cite[\S 2]{R-W}). On the other hand, the study of almost everywhere convergence is much more subtle. The first result
on this subject seems to date back to the paper \cite{Jes} of Jessen. Jessen's theorem states that if $(n_k)$ is a sequence
of positive integers such that $n_k$ divides $n_{k+1}$ for every $k$ then, for any $f\in L^1(\T)$, we have, for almost every $x\in \T$,
 \begin{equation}\label{divide-sum}
\lim_k R_{n_k}f(x) =  \int_{0}^1 f(t) d\rho(t).
\end{equation}
Soon after, Marcinkiewicz and Zygmund \cite{M-Z} on one hand, and Ursell \cite{Urs} independently, gave examples of functions $f\in L^1(\T)$ for which  (\ref{Riem-sum}) fails
almost everywhere. For instance, given $\frac{1}{2} < \delta < 1$,
$$f : x\in ]0,1] \mapsto \abs{x}^{-\delta},$$
 is such an example (see \cite{Urs, Rud}).  
Later, Rudin \cite{Rud} was even able to provide many examples of bounded measurable functions $f$ (characteristic functions indeed)
such that, for almost every $x\in \T$, the sequence $(R_nf(x))$ diverges. Moreover, Rudin's paper highlighted deep connections between
pointwise convergence of Riemann sums along a given subsequence $(n_k)$ of integers and arithmetical properties of the subsequence, a question now widely developed.
The following  different important question has also been considered by many authors :  under which kind of re\-gu\-larity  conditions on $f$
does the associated sequence $(R_n(f))$ of Riemann sums converges {\it a.e.} (see  \cite{R-W} for these questions and many related ones).

 In this paper, we deal with another sort of problem, namely
we study possible extensions of Jessen's result to general locally
compact groups and  dynamical systems.

Let us first come back to Jessen's theorem and give another formulation of this result. Denote by $G$ the group $\T$
and set $G_k = \Z/(n_k\Z)$. Then $(G_k)$ is an increasing sequence of closed subgroups of $G$, whose union
is dense in $G$. If $\rho_k$ is the Haar probability measure on $G_k$, Jessen's result
reads as follows~: for every $f\in L^1(\T)$,
$$\lim_{k\to \infty} \int_{G_k} f(t+x) d\rho_k(t) = \int_G f d\rho \quad a.e.$$

Under this form, this theorem has been extended by Ross and Str\"omberg to locally compact groups. An  assumption about the
behaviour of the modular functions of the subgroups is needed. It is automatically satisfied in the abelian case (see \cite{R-S} and Corollary
\ref{R-S} below). 

More generally, we are interested in this paper by the following questions. {\em Let $G\actson (X,\mu)$ be
an action of a locally compact group $G$ on a measured space $(X,\mu)$, where $\mu$ is $\sigma$-finite, and let $(G_n)_{n\in \N}$ be
an increasing sequence of closed subgroups of $G$, with dense union :
 \begin{itemize}
\item[(i)]  find conditions on the action and on the (right) Haar measures
$\rho_n$ of $G_n$, $\rho$ of $G$, so that
 for every $f\in L^1(X,\mu)$ and every $n$, $t\in G_n \mapsto f(tx)$ is $\rho_n$-integrable for almost every $x\in X$ ;
\item[(ii)] if (i) holds, study the pointwise convergence of the sequence of orbital integrals
$$ \int_{G_n} f(tx) d\rho_n(t).$$
\item[(iii)] identify the pointwise limit, in case it exists.
\end{itemize}}

A necessary condition for (i) to be satisfied is that, for every Borel subset $B$ of $X$ with $\mu(B)<+\infty$, for every $n\in \N$, and for
almost every $x\in X$,
$$\rho_n(\set{t\in G_n\,;\, x\in t^{-1} B}) = \int_{G_n} \un_B(tx) d\rho_n(t) <+\infty.$$

When $\mu$ is finite, say a probability measure, condition (i) implies that the groups $G_n$ are compact.
If we normalize their Haar measures
 by $\rho_n(G_n) = 1$,  the reversed martingale theorem implies that, for $f\in L^1(X,\mu)$, the sequence of orbital integrals converges pointwise and
 in mean (see \cite[Cor. 3.5, page 219]{Tem} and Proposition \ref{compact} below). This fact is well-known. Moreover,
 the limit is the conditional expectation of $f$ with respect to the $\sigma$-field
 of $G$-invariant Borel subsets of $X$. Of course, Jessen's theorem is a particular case. 

When $\mu$ is only $\sigma$-finite, we cannot use the reversed martingale theorem any longer and we shall
need other arguments. As already said,
Ross and Str\"omberg studied the case of the left action $G\actson (G,\rho)$. They normalized the right Haar measures
$\rho_n$ and $\rho$ in such a way that for every $f\in \cC_c(G)$ (the space of continuous functions with compact
support on $G$), we have $\lim_n \rho_n(f) = \rho(f)$. This is always possible, due to a result of Fell (see \cite[Chap. VIII, \S 5]{Bour},
and \cite{R-S} for more references). Morever, Ross and Str\"omberg assumed that for every $n$, the modular function of $G_n$
is the restriction of the modular function of $G$. We shall name this property {\it the modular condition} (MC). Under these assumptions,
Ross and Str\"omberg proved that for every $f\in L^1(G,\rho)$, one has, for almost every $x\in G$,
$$\lim_n \int_{G_n} f(tx) d\rho_n(t) = \int_{G} f(tx) d\rho(t) = \int_{G} f(t) d\rho(t)
\footnote{In this paper we adopt the following convention : each time we write an integral $\int \!\!f $, either $f$ is non-negative, or it is implicitely
contained in the statement that $f$ is integrable.}
.$$
Later, Ross and Willis \cite{Ross-W} provided an example showing that the modular condition does not always hold.
They also proved  that the Ross-Str\"omberg theorem always fails when the modular condition is not satisfied. 

Therefore, in   our paper {\sl we shall always assume that the increasing sequence $(G_n)$ of closed subgroups of $G$ has
a dense union and satisfies the modular condition (MC)}. We shall denote by $\Delta$ the modular function of $G$,
so that $\lambda = \Delta \rho$, where $\lambda$ is the left Haar measure of $G$. {\sl We also assume that
$G$ acts on $(X,\mu)$ in such a way that $\mu$ is $\Delta$-relatively invariant under the action}, in the sense that $s\mu = \Delta(s) \mu$ for
$s\in G$. An important example is the left action $G\actson (G, \rho)$. 

Under these assumptions,  we give a necessary and sufficient condition for  the pointwise limit theorem to
be satisfied.

\begin{thm*}[\bf\ref{lim-3}] The two following properties are equivalent :
\begin{itemize}
\item[(a)] $X$ is a countable union of Borel subsets $B_k$ of finite measure, such that for every $k$ and almost every
$x\in X$, we have
$$\rho(\set{t\in G\,;\,x\in t^{-1}B_k}) = \int_G \un_{B_k}(tx) d\rho(t) <+\infty.$$
\item[(b)]
For every $f\in L^1(X,\mu)$ and for almost every $x\in X$, 
$$\lim_n \int_{G_n} f(tx) d\rho_n(t) = \int_{G} f(tx) d\rho(t).$$
\end{itemize}
\end{thm*}

A crucial intermediate step is the above mentioned Ross-Str\"omberg theorem. For completeness,   
we provide a proof of this result, partly based on one of the ideas contained in \cite{R-S}.

We give examples where our theorem \ref{lim-3} applies. In particular, as an easy consequence we get our second
main result :

\begin{thm*}[\bf\ref{lattice}] Consider $G\actson (X,\mu)$, where now the $\sigma$-finite measure $\mu$ is invariant. 
Let $(G_n)_{n\in \N}$ be an increasing sequence of lattices in $G$, whose union is dense. We fix a Borel fundamental
domain $D$ for $G_0$ and we normalize $\rho$ by $\rho(D) = 1$ \footnote{Note that our assumptions imply the unimodularity
of $G$.}. Then, for every $G_0$-invariant function $f\in L^1(X,\mu)$ and for almost every $x\in X$, we have
$$\lim_n \frac{1}{\abs{G_n \cap D}} \sum_{t\in G_n \cap D} f(tx) = \int_D f(tx) d\rho(t),$$
where $\abs{G_n \cap D}$ is the cardinal of $G_n\cap D$.
\end{thm*}

This gives a simple way to extend a result of Civin \cite{Civ}, which treated the case $G = \R$ by a different method,
apparently not  directly adaptable to  more general locally compact groups $G$.

\section{Notation and conventions}
\setcounter{equation}{0}
\renewcommand{\theequation}{\thesection.\arabic{equation}}
In this paper,  locally compact spaces are implicitely assumed to be Hausdorff and $\sigma$-compact. A
measured space $(X,\mu)$ is a Borel standard space equipped with a (non-negative) $\sigma$-finite measure $\mu$.

Let $G$ be a locally compact group. We denote by $\Delta$ its modular function, and by $\lambda$ and $\rho$ respectively its left and right
Haar measures, so that $\lambda = \Delta \rho$. By an action of $G$ on a measured space $(X,\mu)$
we  mean a Borel map $G\times X \to X$, $(t,x)\mapsto tx$, which is a left action and leaves $\mu$ quasi-invariant. In fact we shall need the following
stronger property :

\begin{defn}\label{rel-inv} Given an action $G\actson (X,\mu)$, we say that $\mu$ is $\Delta${\it -relatively invariant} if $s\mu = \Delta(s)\mu$ for all $s\in G$.
\end{defn}

Note that this property is satisfied for the left action $G \actson (G,\rho)$.

In all our statements, we shall consider an increasing sequence $(G_n)$ of closed subgroups of $G$ with dense union. Then
$\Delta_n, \rho_n, \lambda_n = \Delta_n \rho_n$ will be the modular function and Haar measures of $G_n\,$, respectively. We shall have to choose
appropriate normalizations of $\rho$ and $\rho_n$, $n\in \N$. For a compact group, we usually choose its Haar measure to have
total mass one (but see remark \ref{choix} below).

As already mentioned in the introduction, there is also a natural normalization of the Haar measures as follows (see \cite[Chap. VIII, \S 5]{Bour},
and \cite{R-S}).

\begin{defn} Let $G$ be a locally compact group and $\rho$ a right Haar measure on $G$. Let $(G_n)$ be an increasing sequence of closed subgroups,
whose union is dense in $G$.  There is an essentially unique normalization of the right Haar measures $\rho_n$ of the $G_n$  
such  that, for every continuous function
$f$ on $G$ with compact support, we have $\lim_n \rho_n(f) = \rho(f)$. In this case,
we shall say that the sequence $(\rho_n)$ of  right Haar measures is {\it normalized with respect to $\rho$}.
We shall also say that it is a {\it Fell normalization}.
\end{defn}

 \begin{rem}\label{choix} When $G$ is compact, the Fell normalization is the  classical normalization, where  the Haar measures are probability measures.
 On the other hand, when the $G_n$ are compact whereas $G$ is not, it is easily seen that the Fell normalization implies that
 $\lim_n \rho_n(G_n) = +\infty$. For instance, let $G$ be a countable discrete group which is the union of an increasing sequence
 $(G_n)$ of finite subgroups ({\it e.g.} the group $S_\infty$ of finite permutations of the integers). Then, if $\rho$ is the counting
 measure on $G$, the normalization of the sequence $\rho_n$ with respect to $\rho$ is the sequence of counting
 measures, for which we have $\rho_n(G_n) = \abs{G_n}$, the cardinal of $G_n$.
\end{rem}

Finally, another property of the sequence $(G_n)$ will be fundamental in this paper. It was already present in the work
of Ross and Str\"omberg \cite{R-S} and later shown  to be crucial (see  \cite{Ross-W}).

\begin{defn} Let $G$ be a locally compact group and $(G_n)$ an increasing sequence of closed subgroups of $G$. We say that $(G_n)$
satisfies the {\it modular condition} (MC) if, for every $n$, the modular
function $\Delta_n$ of $G_n$ is the restriction to $G_n$ of the modular function $\Delta$ of $G$. 
\end{defn}

Let us explain the interest of this condition. Given an action $G\actson (X,\mu)$   leaving  $\Delta$-relatively invariant the measure $\mu$,
  the following property of the action holds~: for every Borel
function $f : X\times G \to \R^*_{+}$ (or any $\mu\otimes \lambda$-integrable function $f : X\times G \to \C$), we have
\begin{align}\label{eq-crucial}
 \int_{X\times G} f(tx,t) d\mu(x) d\rho(t) &=\int_{X\times G} f(x,t) d\mu(x) d\lambda(t)\notag \\
 & =  \int_{X\times G} f(x,t^{-1}) d\mu(x) d\rho(t).
\end{align}

This property will be essential throughout this paper, and we shall need it to remain satisfied for all the restricted actions $G_n \actson (X,\mu)$. 
This requires that the restriction of $\Delta$ to $G_n$ is the modular function of $G_n$.

\section{Limit theorems}
\setcounter{equation}{0}
\renewcommand{\theequation}{\thesection.\arabic{equation}}
\subsection{Compacity assumptions}
In this section, the Haar measure of every compact group will have total mass one.
\begin{prop}\label{compact}
Let $G$ be a locally compact group, acting in a measure preserving way on a  probability  space $(X,\mu)$.
Let $(G_n)$ be an increasing sequence of compact subgroups of $G$, whose union is dense in $G$. Let $f\in L^1(X,\mu)$.
\begin{itemize}
\item[(a)] For a.e. $x\in X$, we have $\lim_{n\to \infty} \int_{G_n} f(tx) d\rho_n(t) = \E(f|\cI)(x)$, where $\E(f|\cI)$ is the conditional expectation of $f$
with respect to the $\sigma$-field $\cI$ of $G$-invariant Borel subsets of $X$.
\item[(b)] If morever $G$ is compact, then $\E(f|\cI)(x) = \int_G f(tx) d\rho(t)$ for almost every $x\in X$.
\end{itemize}
\end{prop}

\begin{proof} (a) Observe first that $t\in G_n \mapsto f(tx)$ is $\rho_n$-integrable for almost every $x$, since
$$\int_X \Big(\int_{G_n}\abs{f(tx)} d\rho_n(t)\Big)d\mu(x) = \int_X \abs{f(x)} d\mu(x) <+\infty.$$
We set $R_n(f)(x) = \int_{G_n} f(tx) d\rho_n(x)$. Obviously, this function is $G_n$-invariant and it is $\mu$-integrable. Moreover,
let $A$ be a $G_n$-invariant Borel subset of $X$. Then we have
\begin{align*}
\int_A R_n(f)(x) d\mu(x) & = \int_{X\times G_n} \un_A(x) f(tx) d\mu(x) d\rho_n(t)\\
& = \int_{X\times G_n} \un_A(t^{-1}x) f(x) d\mu(x) d\rho_n(t)\\
&= \int_A f(x) d\mu(x),
\end{align*}
since $A$ is $G_n$-invariant.
Therefore $R_n(f)$ is the conditional expectation of $f$ with respect to the $\sigma$-field $\cB_n$ of
Borel $G_n$-invariant subsets of $X$. The sequence $(\cB_n)$ of $\sigma$-fields is decreasing. The
reversed martingale theorem \cite[page 119]{Nev} states that the sequence $(R_n(f))$ converges $\mu$-{a.e.}
and in mean to the conditional expectation of $f$ with respect to the $\sigma$-field $\cB_\infty = \cap_n \cB_n$ 
of $(\cup_n G_n)$-invariant Borel subsets of $X$. Since  $\cup_n G_n$ is dense in $G$ and $\mu$ is finite,
$\cB_\infty$ is also the $\sigma$-field of $G$-invariant Borel subsets of $X$.

The proof of (b) is straightforward.
\end{proof}

\begin{rem} Let $\mu$ be a $\sigma$-finite measure on a Borel space $(X,\cB)$, and let $(\cB_n)$ be a decreasing sequence
of $\sigma$-fields. Assume that $(X,\cB_n, \mu)$ is $\sigma$-finite for every $n$.  For $f\in L^1(X,\cB,\mu)$, the conditional
expectation $\E(f|\cB_n)$ is well defined, as a Radon-Nikod\'ym derivative. In \cite{Jer}, Jerison has proved that
$\lim_{n\to\infty} \E(f|\cB_n)(x)$ exists almost everywhere. Moreover, if $(X,\cB_\infty,\mu)$ is $\sigma$-finite,
the limit is $\E(f|\cB_\infty)$. Otherwise, one may write $X$ as the disjoint union of two elements $V,W$ of $\cB_\infty$, where $V$ is a countable union of elements of
$\cB_\infty$ of finite measure while any subset of $W$ that belongs to $\cB_\infty$ has measure $0$ or $\infty$. By \cite[\S 2.6]{Jer}, 
$\lim_{n\to\infty} \E(f|\cB_n)(x) = 0$ {\it a.e.} on $W$.

This observation can be used to prove that proposition \ref{compact} still holds under the weaker assumption that $\mu$
is $\sigma$-finite. Indeed, it is enough to prove that when $H$ is any compact group acting in measure preserving way on a $\sigma$-finite
measured space $(X,\cB,\mu)$, then $(X,\cB_H,\mu)$ is still $\sigma$-finite, where $\cB_H$ is the $\sigma$-field of Borel
$H$-invariant subsets. To show this fact, consider a  strictly positive function $f\in L^1(X,\cB, \mu)$. As seen in the proof of proposition \ref{compact},
the function $x\mapsto R(f)(x) = \int_H f(tx) d\rho(t)$ is $H$-invariant and $\mu$-integrable and we deduce 
that $(X,\cB_H,\mu)$ is $\sigma$-finite from the fact that $R(f)$ is strictly positive everywhere.
 Finally, for every  $f\in L^1(X,\cB, \mu)$, the conditional expectation $\E(f|\cB_H)$ 
may be defined and we have obviously $R(f)(x) = \E(f|\cB_H)$.

We shall give another proof of proposition \ref{compact} (b), for a compact group $G$ and a $\sigma$-finite measure $\mu$, in corollary \ref{compact1}.
\end{rem}

In the rest of the paper we are interested in the more general situation where $\mu$ is a $\sigma$-finite measure on $X$, and where the subgroups $G_n$
are not assumed to be compact. In particular $G$ is not always unimodular. 

\subsection{General case : local results} Let $G\actson (X,\mu)$ be a measured $G$-space. When the measure $\mu$ is not finite, it may be
useful to study the restriction of the action to every Borel subset $B$ of $X$ such that $\mu(B)<\infty$,
even if $B$ is not $G$-invariant\footnote{The restriction of $\mu$ to the Borel subspace $B$ will be denoted by the same letter.}.
We extend to $X$ every function defined on $B$, by giving it the value $0$ on
$X\setminus B$. In particular, we have $L^1(B,\mu) \subset L^1(X,\mu)$.

 Note that Borel functions of the form
$$x\mapsto \rho(\set{t\in G\,;\, tx \in B}) = \int_G \un_B(tx) d\rho(t)$$
are $G$-invariant. This crucial fact will be used repeatedly and without mention in the sequel.
 
\begin{thm}\label{lim-2} Let $(G_n)$ be an increasing sequence of closed subgroups of a locally compact group $G$, with dense union
and satisfying condition \hbox{(MC)}. Let
$G\actson (X,\mu)$ be an action leaving the measure $\Delta$-relatively invariant.
 Let $B$ be a Borel subset
of $X$ with $\mu(B) <+\infty$. For a.e. $x\in B$ and every $n\in \N$,  
we assume
that $0< \rho_n(\set{t\in G_n\,;\, tx\in B}) < +\infty$, where $\rho_n$ is a right Haar measure on $G_n$. Then, for every $f\in L^1(B,\mu)$, the averaging sequence of functions
$$x\in B \mapsto \frac{1}{\rho_n(\set{t\in G_n\,;\, tx\in B})} \int_{G_n} f(tx) d\rho_n(t)$$ converges, almost everywhere on $B$ and in $L^1$ norm,
to an element of $L^1(B,\mu)$.
\end{thm}

\begin{proof} We first check that for every $n$ and almost every $x\in B$, the function $ t\in G_n \mapsto  f(tx)$ is $\rho_n$-integrable.
This is a consequence of the following computation (where we use equality (\ref{eq-crucial}))~:
\begin{align*}
\int_B  \frac{1}{\rho_n(\set{s\in G_n\,;\, sx\in B})}&\Big(\int_{G_n} \abs{f(tx)} d\rho_n(t)\Big) d\mu(x) \\
& = \int_{X\times G_n} \un_B(x) \frac{\abs{f(tx)}}{\rho_n(\set{s\in G_n\,;\, sx\in B})}d\mu(x) d\rho_n(t)\\
& =  \int_{X\times G_n} \un_B(tx)\frac{\abs{f(x)}}{\rho_n(\set{s\in G_n\,;\, sx\in B})}d\mu(x) d\rho_n(t)\\
& =  \int_{B} \abs{f(x)}\frac{1}{\rho_n(\set{s\in G_n\,;\, sx\in B})} \Big(\int_{G_n} \un_B(tx) d\rho_n(t)\Big)d\mu(x)\\
& =  \int_{B} \abs{f(x)} d\mu(x)<+\infty.
\end{align*}
In addition, we see that $\displaystyle x\in B \mapsto \frac{\int_{G_n}f(tx)d\rho_n(t)}{\rho_n(\set{s\in G_n\,;\,sx\in B})}$ is $\mu$-integrable on 
 $B$. Denote by $R_n(f)$ this function defined on $B$.  

Let $\cO_n(B)$ be the equivalence relation on $B$ induced by the $G_n$-action : for $x, y \in B$, $x\sim_{\cO_n(B)} y$ if there exists $t\in G_n$
with $x= ty$. We denote by $\cB_n(B)$ the $\sigma$-field of Borel subsets of $B$ invariant under this equivalence
relation. Observe that $R_n(f)$ is invariant under $\cO_n(B)$. It is also straightforward to check that $R_n(f)$ is the conditional
expectation of $f$ with respect to $\cB_n(B)$. Then, again the conclusion follows from the reversed martingale theorem.
\end{proof}

 Concerning the pointwise convergence of sequences of orbital integrals, we immediately get :
  
 \begin{cor}\label{local} Under the assumptions of the previous theorem,
 the following conditions are equivalent :
 \begin{itemize}
 \item[(a)] $\lim_{n\to \infty}  \rho_n(\set{t\in G_n\,;\,t x\in B})$ exists a.e. on $B$ ;
 \item[(b)] for every $f\in L^1(B,\mu)$, $\lim_{n\to \infty}\int_{G_n}f(tx)d\rho_n(t)$ exists a.e. on $B$.
 \end{itemize}
 \end{cor}
 
 \subsection{General case : global results}
 
 In order to study the problem globally, we shall need the following lemma, insuring the integrability
 of orbital functions.
 
\begin{lem}\label{lem-conv} Let $G\actson (X,\mu)$ be an action leaving the measure $\Delta$-relatively invariant. The two following conditions are equivalent :
\begin{itemize}
\item[(i)] $X = \cup B_k$, where every $B_k$ is a Borel subset of $X$, with $\mu(B_k) < +\infty$, such that 
for almost every $x\in X$, 
$$\rho(\set{t\in G\,;\, tx\in B_k}) <\infty.$$
\item[(ii)] For every $f\in L^1(X,\mu)$, $f_x : t\mapsto f(tx)$
belongs to $L^1(G,\rho)$ for a.e. $ x\in X$.
\end{itemize}
\end{lem}

\begin{proof} (ii) $\Rightarrow$ (i) is obvious : write $X$ as a countable union of Borel subsets $B_k$ of finite measure and take $f= \un_{B_k}$.

(i) $\Rightarrow$ (ii). Let $f\in L^1(X,\mu)_+$. It suffices to show that for every $k$ and {\it a.e.} $x\in B_k$,
the function $t\mapsto f(tx)$ is $\rho$-integrable. We set
$$ A = \set{x\in X\,;\,  \int_G \un_{B_k}(tx) d\rho(t) = 0}.$$
Note that $A$ is $G$-invariant.

We first check that for {\it a.e.} $x\in B_k \cap A$, we have $\int_G f(tx) d\rho(t) = 0$.
Indeed,
\begin{align*}
\int_{B_k\cap A}\int_G f(tx) d\rho(t) d\mu(x) & =
\int_{X\times G} \un_{B_k \cap A}(tx) f(x) d\rho(t) d\mu(x) \\
&= \int_X f(x) \Big(\int_G  \un_{B_k \cap A}(tx)d\rho(t)\Big) d\mu(x) = 0.
\end{align*}
The first equality uses relation (\ref{eq-crucial}) and the last one follows from the observation that $\un_{B_k \cap A}(tx)\not= 0$ implies
$x\in A$ since $A$ is $G$-invariant. Hence, we get  $\int_G f(tx) d\rho(t) = 0$ {\it a.e.} on $B_k \cap A$.

Now let us consider $\displaystyle\int_{B_k\setminus A} \frac{1}{\rho(\set{s\in G\,;\, sx\in B_k})}\Big(\int_G f(tx) d\rho(t)\Big) d\mu(x)$.
This integral is equal to
\begin{align*}
&\int_{X\times G} \un_{B_k\setminus A}(x)\frac{1}{\rho(\set{s\in G\,;\, sx\in B_k})} f(tx) d\mu(x)d\rho(t) \\
&= \int_{X\times G} \un_{B_k\setminus A}(tx)\frac{1}{\rho(\set{s\in G\,;\, sx\in B_k})} f(x)  d\mu(x)d\rho(t)\\
&=\int_{X\setminus A} f(x)\Big(\int_G\frac{\un_{B_k\setminus A}(tx)}{\rho(\set{s\in G\,;\, sx\in B_k})}d\rho(t)\Big) d\mu(x)\\
&\leq \int_{X\setminus A} f(x) d\mu(x) <+\infty,
\end{align*}
since,  for every $x$ such that $tx \in B_k\setminus A$ we have $x\in G(B_k\setminus A) = GB_k\setminus A$, and
$$\int_G \un_{B_k\setminus A}(tx) d\rho(t) \leq \int_G \un_{B_k}(tx) d\rho(t) = \rho(\set{s\in G\,;\, sx\in B_k}),$$
with
$$0<  \rho(\set{s\in G\,;\, sx\in B_k}) <+\infty\quad a.e. \quad\text{on}\quad X\setminus A.$$

It follows that  $\int_G f(tx) d\rho(t) <+\infty$ for amost every $x \in B_k\setminus A$.
\end{proof}

\begin{rem}\label{action} The assumption of this lemma holds for instance when $G$ acts on $(X,\mu)$, where  $\mu$
is a  $\Delta$-relatively invariant  Radon measure on a locally compact space $X$, 
the action being continuous with closed orbits and compact stabilizers.
Indeed, in this situation, for $x\in X$, the natural map $G/G_x \to Gx$, where $G_x$ is the stabilizer of $x$, is an homeomorphism.
Then if $B$ is an open relatively compact subset of $X$, the set $\set{t\in G\,;\, tx\in B}$ is  open and relatively compact
in $G$ and the conclusion follows. Particular cases are proper actions, and more generally integrable actions \cite{Rie}.
\end{rem}

We shall now give a condition sufficient to guarantee the pointwise convergence of sequences of orbital integrals for every $f\in L^1(X,\mu)$.
Note that the integrability of the functions appearing in the statement below follows from lemma \ref{lem-conv}.

\begin{thm}\label{lim-1} Let $(G_n)$ be an increasing sequence of closed subgroups in $G$, with dense union
and satisfying condition \hbox{(MC)}. Let
$G\actson (X,\mu)$ be an action leaving the measure $\Delta$-relatively invariant.
 We assume that $X = \cup X_k$ where $(X_k)$ is an increasing
sequence of Borel subspaces, such that for all $k$,
\begin{itemize}
\item[(i)] $\mu(X_k) <+\infty$ ;
\item[(ii)] $\rho_n(\set{t\in G_n\,;\, tx\in X_k})>0$ for almost every $x\in X_k$  and every $n$ ;
\item[(iii)] there exists $c_k >0$ such that for  almost every  $x\in X$,
 $$\sup_{n} \rho_n(\set{t\in G_n\,;\, tx\in X_k}) \leq c_k
  .$$
 \end{itemize}
 The following conditions are equivalent : 
 \begin{itemize}
 \item[(a)] for every $k$,  $\lim_n \rho_n(\set{t\in G_n\,;\, tx\in X_k})$ exists  for a.e. $x \in X_k$ ;
\item[(b)]  the pointwise limit $\lim_n \int_{G_n} f(tx)d\rho_n(t)$ exists a.e. on $X$, for every $f\in L^1(X,\mu)$ ;
 \item[(c)] there exists a dense subset $\cD$  of $L^1(X,\mu)$ such that the pointwise limit $\lim_n \int_{G_n} f(tx)d\rho_n(t)$ exists a.e. on $X$, for every $f$ in $\cD$.
 \end{itemize}
 \end{thm}
 
 \begin{rem}\label{en-pratique}  Assume that for every $x\in X_k$ and every $n$,
\begin{equation}\label{pratique}\rho_n(\set{t\in G_n\,;\, tx\in X_k}) \leq c_k.
\end{equation}
 Then, assumption (iii) of the above theorem is fulfilled. Indeed, by invariance, if (\ref{pratique}) holds for $x\in X_k$,
 it also holds for $x\in G_n X_k$. On the other hand, if $x\notin G_n X_k$ then 
 $$\set{t\in G_n\,;\, tx\in X_k} =\emptyset\,,$$
 and therefore $\rho_n(\set{t\in G_n\,;\, tx\in X_k}) = 0 \leq c_k$.
  \end{rem}
 
 For the proof of theorem \ref{lim-1}, we need the following lemma which repeats arguments from \cite[Lemma 3]{R-S}.
 
 \begin{lem}\label{lem-ine-max}  Let $(G_n)$ be an increasing sequence of closed subgroups in $G$, with dense union
and satisfying condition \hbox{(MC)}. Let
$G\actson (X,\mu)$ be an action leaving the measure $\Delta$-relatively invariant. Let $X_k\subset X$ satisfying conditions
(i) and (iii) of the previous theorem and let $f : X \to \R_+$ be a Borel function. For $x\in X$, set 
 $$f^\star(x) = \sup_n \int_{G_n} f(tx)d\rho_n(t),$$ and
 for $\alpha >0$, set $Q_\alpha = \set{x\in X \,;\, f^\star(x) >\alpha}$. Then  we have

 \begin{equation}\label{ine-max}
 \alpha\mu(Q_\alpha \cap X_k) \leq c_k \int_{Q_\alpha} f d\mu.
 \end{equation}
 \end{lem}
 
 \begin{proof} For $n\in \N$, we set $\phi_n(x) =\int_{G_n} f(tx)d\rho_n(t)$ and we introduce the subsets
 \begin{align*}
 E_n &= \set{x\in X\,;\, \phi_n(x) >\alpha}\\
D_n & = \set{x\in X\,;\, \sup_{1\leq l \leq n} \phi_l(x) > \alpha}. 
\end{align*}
We fix an in integer $N$. It is enough to show that
$$\alpha\mu(D_N \cap X_k) \leq c_k \int_{D_N} f d\mu$$ 
since $Q_\alpha$
 is the increasing union of the sets $D_N$, $N\geq 1$. 

We decompose $D_N$ as the union $\dst \cup_{n= 1}^N F_n$, where the subsets $F_n$ are mutually disjoint, and defined by
$$F_n = E_n \cap (\bigcup_{l= n+1}^N E_{l}^c).$$

Observe that $G_n E_n = E_n$, and therefore $G_n F_n = F_n$ for every $n$. We have
\begin{align*}
\alpha\mu(F_n \cap X_k) & \leq \int_{F_n\cap X_k} \phi_n(x) d\mu(x) \\
& \leq \int_{X\times G_n} \un_{F_n\cap X_k}(x) f(tx) d\mu(x) d\rho_n(t)\\
&\leq \int_{X\times G_n} \un_{F_n\cap X_k}(tx) f(x) d\mu(x) d\rho_n(t)\\
& \leq \int_X f(x)\Big(\int_{G_n}\un_{F_n\cap X_k}(tx) d\rho_n(t)\Big) d\mu(x).
\end{align*}
For $\un_{F_n\cap X_k}(tx)$ to be non-zero, it is necessary that $x\in G_n (X_k\cap F_n) = G_n X_k \cap F_n$. It follows that
\begin{align*}
\alpha\mu(F_n \cap X_k) 
&\leq \int_{ F_n} f(x)\Big(\int_{G_n}\un_{F_n\cap X_k}(tx) d\rho_n(t)\Big) d\mu(x)\\
&\leq c_k \int_{F_n} f(x) d\mu(x).
\end{align*}
The conclusion is then an immediate consequence of the fact that $D_N$ is the union of the mutually disjoint subsets
$F_n$, $1\leq n \leq N$.
\end{proof}

\begin{rem}\label{ine-seule} As a particular case, we shall use the following assertion. Let $G\actson (X,\mu)$ be a $G$-action such
that $\mu$ is $\Delta$-relatively invariant. Let $X_k$ be a
 Borel subset of $X$ with $\mu(X_k) <+\infty$.  
Assume  the existence of $c_k$ such that
for almost every $x\in X$,  $\rho(\set{t\in G\,;\, tx\in X_k}) \leq c_k$.  Let $f : X\to \R_+$ be a Borel function. For $\alpha >0$,
set $\tilde{Q}_\alpha = \{x\in X\,;\, \int_G f(tx) d\rho(t) > \alpha\}$. Then,  we have
$$\alpha\mu(\tilde{Q}_\alpha \cap X_k) \leq c_k \int_{\tilde{Q}_\alpha} fd\mu.$$
\end{rem}

\begin{proof}[Proof of theorem \ref{lim-1}] (b) $\Rightarrow$ (a) is obvious. Let us show that (a) $\Rightarrow$ (c). Let $f\in L^1(X,\mu)$,
null outside $X_k$. Fix $p> k$. By theorem \ref{lim-2}, we know that 
 $$\lim_n \frac{\int_{G_n} f(tx) d\rho_n(t)}{\rho_n(\set{s\in G_n\,;\, sx\in X_p})} \quad\mbox{exists}\quad a.e. \,\,\,\mbox{on} \quad X_p.$$
If (a) holds, we immediately get the existence of $\lim_n \int_{G_n} f(tx) d\rho_n(t)$ almost everywhere on $X_p$ and therefore on the union $X$ of the $X_p$.
Now, observe that such functions $f$, supported in some $X_k$, form a dense subspace of $L^1(X,\mu)$.

Finally, let us prove that (c) implies (b). We introduce
$$\Lambda(f)(x) = \lim_{N\to +\infty} \Big(\sup_{n,m\geq N} \abs{\phi_n(f)(x) - \phi_m(f)(x)}\Big).$$
We fix an integer $p$, and we shall show that $\Lambda(f)(x) = 0$ for almost every $x\in X_p$. This will end the proof.
Take $g\in \cD$. We have $\Lambda(g) =0$ {\it a.e.} on $X$ and
$$\Lambda(f) = \Lambda(f) -\Lambda(g) \leq \Lambda(f-g) \leq 2 \abs{f-g}^\star.$$
Given $\alpha >0$, it follows from lemma \ref{lem-ine-max}, that
\begin{align*}
\mu(\set{x\in X_p\,;\,\Delta(f)(x) >\alpha})& \leq \mu(\set{x\in X_p\,;\, \abs{f-g}^\star >\alpha/2})\\
&\leq \frac{2c_p}{\alpha} \norm{f-g}_1.
\end{align*}
Since we can choose $g$ so that $ \norm{f-g}_1$ is as close to $0$ as we wish, we see that $\mu(\set{x\in X_p\,;\, \Delta(f)(x) >\alpha}) = 0$, from
which we get $\mu(\set{x\in X_p\,;\,\Delta(f)(x) >0}) = 0$.
\end{proof}

 We apply theorem \ref{lim-1} to the following situation where, in addition, it is possible to identify the limit.

\begin{thm}\label{lim-propre} Let $G$ act properly on a locally compact $\sigma$-compact space $X$
and   let $\mu$ be a $\Delta$-relatively invariant Radon measure on $X$. Let $(G_n)$ be an increasing sequence
of closed subgroups, whose union is dense in $G$ and which satisfies the modular condition (MC).
We assume that the sequence $(\rho_n)$ of Haar measures is normalized with respect to $\rho$. 
Then, for every $f\in L^1(X,\mu)$, we have, for a.e. $x\in X$,
\begin{equation}\label{orbital-lim}
\lim_{n\to \infty} \int_{G_n} f(tx) d\rho_n(t) = \int_G f(tx) d\rho(t).
\end{equation}
\end{thm}

\begin{proof} Let $(X_k)$ be an increasing sequence of open relatively compact subspaces of $X$, with $X = \cup X_k$.
Of course, since the action is proper, we have 
$$0<\rho_n(\set{t\in G_n\,;\, tx\in X_k}) <+\infty$$
 for every $x\in X_k$.
Let us show that condition (iii) of theorem \ref{lim-1} is also fulfilled. Set $K_k = \set{t\in G\,;\, tX_k \cap X_k \not = \emptyset}$.
This set is relatively compact. We choose a continuous function $\varphi$ on $G$, with compact support, such that $\un_{K_k} \leq \varphi$.
We have
$$\forall n\in \N, \forall x\in X_k,\quad \rho_n(\set{t\in G_n\,;\, tx\in X_k}) \leq \rho_n(\varphi).$$
Since $\lim_n \rho_n(\varphi) = \rho(\varphi) <+\infty$, we obtain the existence of a constant $c_k$
such that 
$$\forall n\in \N,\,\, \forall x\in X_k,\quad \rho_n(\set{t\in G_n\,;\, tx\in X_k}) \leq c_k.$$
Now (iii) of theorem \ref{lim-1} is satisfied, by remark \ref{en-pratique}. Clearly, we may also choose $c_k$ such that, as well, $ \rho(\set{t\in G\,;\,tx\in X_k}) \leq c_k$
for all $x\in X$.

The required integrability conditions for (\ref{orbital-lim}) follow from lemma \ref{lem-conv}. The existence of the limit is an
immediate consequence of theorem \ref{lim-1}, applied to  the space  $\cD = \cC_c(X)$ of continuous functions with compact
support in $X$. We use the fact that for every $x\in X$ and $f \in  \cC_c(X)$, the function $t\mapsto f(tx)$ is continuous with compact support.
Hence, by the normalization of the $\rho_n$, we have the existence of $\lim_n \int_{G_n} f(tx) d\rho_n(t)$. Here, we even know that
the limit  is  $\int_G f(tx) d\rho(t)$, for every $x$. 

It remains to identify the limit for every $f\in L^1(X,\mu)$. We set
$$\tilde{\Lambda}(f)(x) = \lim_{N\to\infty}\Big(\sup_{n\geq N} \abs{\int_{G_n} f(tx) d\rho_n(t) - \int_G f(tx) d\rho(t)}\Big).$$
As in the proof of theorem \ref{lim-1}, we fix $p$, and we only need to show that $\tilde{\Lambda}(f)(x) = 0$ for almost every $x\in X_p$.
Let $g$ be a continuous function with compact support  on $X$. We have
\begin{align*}
\tilde{\Lambda}(f)(x) &= \tilde{\Lambda}(f)(x) - \tilde{\Lambda}(g)(x)\\
&\leq \tilde{\Lambda}(f-g)(x)\\
&\leq \lim_N\Big( \sup_{n\geq N} \abs{\int_{G_n}(f(tx) - g(tx)) d\rho_n(t)} \Big)+ \abs{\int_G(f(tx) - g(tx)) d\rho(t)}\\
&\leq \abs{f-g}^\star + \int_G\abs{f(tx) -g(tx)} d\rho(t).
\end{align*}
Given $\alpha >0$, we have
\begin{align*}
\mu(\set{x\in X_p\,;\,\tilde{\Lambda}(f)(x)> \alpha})\leq &\mu(\set{x\in X_p\,;\, \abs{f-g}^\star > \alpha/2})\\
& + \mu(\set{x\in X_p\,;\, \int_G\abs{f(tx) -g(tx)} d\rho(t) >\alpha/2})\\
& \leq \frac{4 c_p}{\alpha} \norm{f-g}_1.
\end{align*}
The last inequality follows from lemma \ref{lem-ine-max} and remark \ref{ine-seule}.
Now, we approximate $f$ by a sequence $(f_n)$ of continuous functions with compact support. This
gives $\mu(\set{x\in X_p\,;\, \tilde{\Lambda}(f)(x)> \alpha}) = 0$. The conclusion is obtained by letting $\alpha$
go to $0$.
\end{proof}

As a particular case, we obtain the following result of Ross and Str\"omberg. In contrast to their proof,
we do not use the theorem of Edwards and Hewitt (\cite[Theorem 1.6]{EH}) on pointwise limits of sublinear operators whose ranges
are families of measurable functions.

\begin{cor}[\cite{R-S}]\label{R-S}  Let $G$ be a locally compact group, and $(G_n)$ be an increasing sequence
of closed subgroups,  whose union is dense in $G$ and which satisfies the modular condition (MC).
We assume that the sequence $(\rho_n)$ of Haar measures is normalized with respect to $\rho$.
Then, for every $f\in L^1(G,\rho)$, we have
$$\lim_n \int_{G_n} f(tx) d\rho_n(t) = \int_{G} f(t) d\rho(t), \quad a.e.$$
\end{cor}
\begin{proof} We apply theorem \ref{lim-propre} to $G\actson (G,\rho)$.
\end{proof}

We can now state our main theorem.

\begin{thm}\label{lim-3}  Let $G\actson (X,\mu)$ be an action on a measured space, leaving the measure $\Delta$-relatively invariant. 
Let $(G_n)$ be an increasing sequence
of closed subgroups, whose union is dense in $G$ and which satisfies the modular condition (MC).
We assume that the sequence $(\rho_n)$ of Haar measures is normalized with respect to $\rho$.   
The two following properties are equivalent :
\begin{itemize}
\item[(a)] $X$ is a countable union of Borel subsets $B_k$ of finite measure, such that for every $k$ and almost every
$x\in X$, we have
\begin{equation}\label{eqlim-3} \rho(\set{t\in G\,;\, tx\in B_k}) = \int_G \un_{B_k}(tx) d\rho(t) <+\infty.
\end{equation}
\item[(b)]
For every $f\in L^1(X,\mu)$ and for almost every $x\in X$, 
$$\lim_n \int_{G_n} f(tx) d\rho_n(t) = \int_{G} f(tx) d\rho(t).$$
\end{itemize}
\end{thm}

\begin{proof} Assumption (b) contains the assertion that for every $f\in L^1(X,\mu)$ and almost every $x\in X$, the function
$f_x : t \mapsto f(tx)$ is $\rho$-integrable. Thus, obviously (b) implies (a).

Let us show that (a) implies (b). Let $f\in L^1(X,\mu)$. By lemma \ref{lem-conv}, there exists a conull subset $E\subset X$, such that for every $x\in E$,
the function $f_x : t\mapsto f(tx)$ is in $L^1(G,\rho)$. We apply to $f_x$ the previous corollary. There exists a conull subset
$A_x$ in $G$, such that for every $s\in A_x$ :
\begin{itemize}
\item[(i)] for $n\in \N$, $t\in G_n \mapsto f_x(ts)$ in $\rho_n$-integrable ;
\item[(ii)] $\lim_n \int_{G_n} f_x(ts) d\rho_n(t) = \int_G f_x(t) d\rho(t)$.
\end{itemize}

Denote by  $D$ the set of all $(s,x)\in G\times X$ for which 
\begin{itemize}
\item $t\in G_n \mapsto f(tsx) = f_x(ts)$ is $\rho_n$-integrable for all $n$,
\item  $t\in G \mapsto f(tsx) = f_x(ts)$ is $\rho$-integrable,
\item $\lim\int_{G_n} f(tsx) d\rho_n(t)) = \int_G f(tsx) d\rho(t)$.
\end{itemize}
Then $D$  is a Borel subset of $G\times X$. Moreover,
$$D \supset \set{(s,x)\,;\, x\in E, s\in A_x}.$$
It follows, by using twice the Fubini-Tonelli theorem, that $D$ is conull, and that for almost every $s\in G$, we have, for almost $x\in X$ :
\begin{itemize}
\item[(1)]  $f_{sx} $ is $\rho_n$-integrable for $n\in \N$, and is $\rho$-integrable ;
\item[(2)]  $\lim_n \int_{G_n} f(tsx) d\rho_n(t) = \int_G f(tx) d\rho(t)$.
\end{itemize}
Choose such a $s$ and let $C(s)$ be a conull subset of $X$ for which (1) and (2) occur. Then for any $y\in sC(s)$, which is also conull,
we have the required properties.
\end{proof}

\begin{cor}\label{compact1} Let $G$ be a compact group acting on a measured space $(X,\mu)$ in such a way that
the $\sigma$-finite measure $\mu$ is invariant. Let $(G_n)$ be an increasing sequence of closed subgroups of $G$, whose union
is dense in $G$. We choose the Haar measures to have total mass $1$. Then for every $f\in L^1(X,\mu)$ we have
$$\lim_{n\to +\infty} \int_{G_n} f(tx) d\rho_n(t) = \int_G f(tx) d\rho(t).$$
\end{cor}

Theorem \ref{lim-3} also applies to several other situations. We have already mentioned in remark \ref{action} the case of continuous
actions with closed orbits and compact stabilizers. We now give another example of application.

\begin{cor} Let $G$ be a  locally compact group acting on a measured space $(X,\mu)$ is such a way that
the $\sigma$-finite measure $\mu$ is invariant. Let $(G_n)$ be an increasing sequence of closed subgroups
of $G$, whose union is dense in $G$ and which satisfies the modular condition (MC).
We assume that the sequence $(\rho_n)$ of Haar measures is normalized with respect to $\rho$. 
 Let $f\in L^1(G\times X, \rho\otimes \mu)$. Then for $\rho\otimes \mu$ almost every $(s,x)\in G\times X$,
we have
$$\lim_{n\to \infty} \int_{G} f(ts, tx) d\rho_n(t) = \int_{G} f(ts,tx) d\rho(t).$$
\end{cor}

\begin{proof} We apply theorem \ref{lim-3} to $B_k = U_k \times V_k$, where $(U_k)_k$ is a sequence
of relatively compact open subsets of $G$ with $\cup U_k = G$, and where $(V_k)_k$ is a sequence of
Borel subsets of $X$, of finite measure, with $\cup V_k = X$. It suffices to observe that
$$\set{t\in G\,;\, t(s,x) \in U_k \times V_k}  \subset \set{t\in G\,;\, ts \in U_k },$$
and
$$\rho(\set{t\in G, t(s,x) \in B_k}) \leq \rho(\set{t\in G\,;\, ts \in U_k }) <+\infty.$$
\end{proof}

\begin{cor}\label{presque}  Let $G\actson (X,\mu)$ and $(G_n)$ be as in the previous corollary. 
Let $h\in L^1(X,\mu)$ and let $E$ be a Borel subset of $G$ with $\rho(E) < \infty$.
Then, for  almost every $s\in G$, we have
$$\lim_{n\to \infty} \int_{Es \cap G_n} h(tx) d\rho_n(t) = \int_{Es} h(tx) d\rho(t).$$
\end{cor}

\begin{proof} We apply the previous corollary with $f(t,x) = \un_{E}(t)f(x)$.
\end{proof}

\begin{ex}\label{example} Let $G = \R$, acting on $\R$ be left translations, and for $n\in \N$, take $G_n = \Z/(2^n\Z)$. The Haar measure on $G_n$ is normalized
by giving the weight $1/(2^n)$ to each point and we take for $\mu$ the Lebesgue measure $\rho$ on $\R$, normalized by $\rho([0,1]) = 1$. Let $E$ be a Borel subset of
$\R$ such that $\rho(E) < +\infty$.  The previous corollary gives that for every $f\in L^1(\R,\rho)$ and for almost every $s\in \R$, we have
$$\lim_{n\to \infty} \frac{1}{2^n} \sum_{\set{k\,;\,k/2^n \in E+s}} f\big(\frac{k}{2^n}+x\big) = \int_{E+s} f(t+x)d\rho(t)\quad a.e.$$ 

Let us take $E = \Q$ for example. For every $s$ irrational, the above equality holds (both sides are $0$). On the other hand, for $s\in \Q$, and
$f = \un_{[0,1]}$, this equality is false for every $x$.
\end{ex}

\subsection{An extension of a theorem of Civin}
We are now interested by the following problem : let $G \actson (X,\mu)$ as before, that is, the $\sigma$-finite measure $\mu$ is $\Delta$-relatively invariant. 
We are given an increasing sequence of closed subgroups of $G$, with dense union, and satisfying the modular condition. Let $E$ be a Borel subset of $X$
with $\rho(E) < \infty$, and let $f\in L^1(X,\mu)$. Find conditions under which 
$$\lim_n \int_{G_n \cap E} f(tx) d\rho_n(t) = \int_E f(tx) d\rho(t)$$
almost everywhere. 

In \cite{Civ}, Civin has considered the particular case where $G= \R$, $G_n = \Z/(2^n\Z)$ as in example \ref{example}. Let $(t,x)\mapsto t+x$ be a measure
preserving action of $\R$ onto a measured space $(X,\mu)$. Civin's result states that for every $f\in L^1(X,\mu)$ such that $f(1+ x) = f(x)$ {\it a.e.}, then
for almost every $x\in X$, 
\begin{align*}
\lim_{n\to \infty} \int_{G_n \cap [0,1[} f(t+x) d\rho_n(t) & = \lim_{n\to \infty} \frac{1}{2^n} \sum_{k = 1}^{2^n} f(\frac{k}{2^n} + x) \\
& = \int_{0}^1 f(t+x) d\rho(t).
\end{align*}

More generally, we have :

\begin{thm}\label{lattice} Let $(G_n)_{n\in \N}$ be an increasing sequence of lattices of $G$ (therefore $G$ is unimodular), with dense union,
and let  $D$ be a fundamental domain for $G_0$. Let $G\actson (X,\mu)$  
be a measure preserving action. We assume that the Haar measure of $G$ is normalized so that the volume of $D$ is $1$.
Let $f\in L^1(X,\mu)$ be such that,  for every $t\in G_0$, $f(tx) = f(x)$ almost everywhere. Then
$$\lim_n \frac{1}{\abs{G_n \cap D}} \sum_{t\in G_n \cap D} f(tx) = \int_D f(tx) d\rho(t)$$
for a.e. $x\in X$.
\end{thm}

\begin{proof} We normalize the Haar measure on $G_n$ by giving to each point the measure $1/\abs{G_n \cap D}$. This gives a normalized
sequence of Haar measures with respect to $\rho$. Corollary \ref{presque} gives that, for almost every $s\in G$, we have
$$\lim_n \frac{1}{\abs{G_n \cap Ds}} \sum_{t\in G_n \cap Ds} f(tx)  = \int_{Ds} f(tx) d\rho(t).$$
For every $t\in G_n \cap Ds$, there exists a unique $g\in G_0$ such that $gt \in D$. Due to the $G_0$-invariance of $f$,
we have $f(tx) = f(gtx)$ and therefore 
$$\frac{1}{\abs{G_n \cap Ds}} \sum_{t\in G_n \cap Ds} f(tx) = \frac{1}{\abs{G_n \cap D}} \sum_{t\in G_n \cap D} f(tx).$$
On the other hand, by \cite[Corollaire, Page 69]{Bour},  the $G_0$-invariance of $f$ implies that the integral $\int_{Ds} f(tx) d\rho(t)$ does not depend on the choice of the fundamental domain :
we  have $\int_{Ds} f(tx) d\rho(t) = \int_D f(tx) d\rho(t)$.
\end{proof}

\bibliographystyle{alpha}
\bibliography{ergobibliography}
\end{document}